\documentclass[12pt]{article}

\usepackage{ae,amsfonts,euscript,enumerate}
\usepackage{amsmath,euscript}
\usepackage{amssymb}
\usepackage{amsthm}
\usepackage{enumerate}
\usepackage{amsfonts}
\usepackage{comment}
\usepackage{colortbl}
\usepackage{a4wide}
\usepackage{amssymb,amsmath,array}

\newtheorem{thm}{Theorem}[section]
\newtheorem{lem}[thm]{Lemma}

\newtheorem{prop}[thm]{Proposition}
\newtheorem{cor}[thm]{Corollary}

\theoremstyle{defn}
\newtheorem{defn}{Definition}

\newtheorem*{notn}{Notation}

\newcommand{\gc}{ [ \hspace{-0.65mm} [}
\newcommand{\dc}{]  \hspace{-0.65mm} ]}

\newcommand{\vb}[1]{
{\boldsymbol #1}
}

\newcommand{\F}{\mathfrak}
\newcommand{\B}{\mathbb}

\newcommand{\hc}{\mathcal{H}}

\newcommand{\U}{U_q^+(\g)}

\newcommand{\spec}{{\rm Spec}}

\newcommand{\prim}{{\rm Prim}}

\newcommand{\g}{\mathfrak{g}}

\newcommand{\comp}{\mathbb{C}}

\def\C{\mathbb{K}}

\def\co{{\mathcal O}}

\def\oqmm13{\co_q(M_{1,3})}
\def\oqm23{\co_q(M_{2,3})}

\input{xy}
\xyoption{all}

\title{Primitive ideals in quantum Schubert cells: dimension of the strata}
\author{J.~Bell\thanks{The first named authors thank NSERC for its generous support.}, K.~Casteels\thanks{The second named authors thank NSERC for its generous support.} ~and S.~Launois\thanks{The third named author research was supported by a Marie Curie European Reintegration Grant within the $7^{\mbox{th}}$ European Community Framework Programme.}}

\long\def\symbolfootnote[#1]#2{\begingroup\def\thefootnote{\fnsymbol{footnote}}\footnote[#1]{#2}\endgroup}

\def\keywords#1{\def\@keywords{#1}}
\def\@keywords{}
\def\subjclass#1{\def\@subjclass{#1}}
\def\@subjclass{}

\date{}

\begin{document}

\maketitle

\subjclass{ {\it Primary}: 16T20,17B37, 20G42; {\it Secondary}: 17B22}
\symbolfootnote[0]{{\it 2010 Mathematics Subject Classification.}\enspace \@subjclass. }


\begin{abstract}
The aim of this paper is to study the representation theory of quantum Schubert cells. Let $\g$ be a simple complex Lie algebra. To each element $w$ of the Weyl group $W$ of $\g$, De Concini, Kac and Procesi have attached a subalgebra $U_q[w]$ of the quantised enveloping algebra $U_q(\g)$. Recently, Yakimov showed that these algebras can be interpreted as the quantum Schubert cells on quantum flag manifolds. In this paper, we study the primitive ideals of $U_q[w]$. More precisely, it follows from the Stratification Theorem of Goodearl and Letzter that the primitive spectrum of $U_q[w]$ admits a stratification indexed by those primes that are invariant under a natural torus action. Moreover each stratum is homeomorphic to the spectrum of maximal ideals of a torus. The main result of this paper gives an explicit formula for the dimension of the stratum associated to a given torus-invariant prime.
\end{abstract}

\section{Introduction}

Let $\g$ be a simple Lie algebra of rank $n$ over the field of complex numbers, and let $\pi:=\{\alpha_1,\dots,\alpha_n\}$ be the set of simple roots associated to a triangular decomposition $\g=\mathfrak{n}^- \oplus \mathfrak{h} \oplus \mathfrak{n}^+$. To any element $w$ in the Weyl group $W$ of $\g$ corresponds a nilpotent Lie algebra $\mathfrak{n}_w:= \mathfrak{n}^+ \cap \mathrm{Ad}_{w}(\mathfrak{n}^-) $, where $\mathrm{Ad}$ stands for the adjoint action. De Concini, Kac and Procesi \cite{DKP} defined a quantum analogue of the enveloping algebra of this nilpotent Lie algebra by using the braid group action of $W$ on the quantised enveloping algebra $U_q(\g)$ induced by Lusztig automorphisms. The resulting (quantum) algebra is denoted by $U_q[w]$.

Recently, Yakimov \cite{Y} has given a geometrical interpretation to these algebras; he has shown that this family of algebras can actually be viewed as quantisations of algebras of regular functions on Schubert cells. This approach has allowed him to use earlier works of Joseph \cite{josephbook,jos} and Gorelik \cite{gorelik} in order to study the primes of $U_q[w]$ that are invariant under the natural action of a torus of rank $n$.

Our aim in this article is to study the representation theory of these algebras $U_q[w]$ that we refer to as {\it quantum Schubert cells}. As usual for infinite-dimensional algebras, it is a very difficult problem, and so we follow Dixmier's approach and study the primitive ideals of these algebras. 

In order to investigate the primitive ideals of various quantum algebras, Goodearl and Letzter have developed a strategy based on the rational action of a torus. Indeed, building on works of Hodges-Levasseur \cite{hl1,hl2}, Hodges-Levasseur-Toro \cite{hlt}, Joseph \cite{josephbook,jos}, Brown-Goodearl \cite{bgtrans}, etc.,  Goodearl and Letzter \cite{gl2} have described a stratification of the prime and primitive spectra of an algebra $A$ supporting a rational action of a torus $H$. In the context of quantum Schubert cells, this stratification of the prime spectrum is  parametrised by those prime ideals that are invariant under this torus action and each stratum is homeomorphic to the prime spectrum of a commutative Laurent polynomial ring over the base field. Torus-invariant primes ideals of $U_q[w]$ have recently been studied by M\'eriaux and Cauchon \cite{CM} on one hand and by Yakimov \cite{Y} on the other hand.  In particular, let us point out that they all proved that torus-invariant primes in $U_q[w]$ are in one-to-one correspondence with the initial Bruhat interval $[\mathrm{id},w]$. Moreover, using earlier works of Joseph and Gorelik,  and the interpretation of $U_q[w]$ as quantum Schubert cells, Yakimov has described efficient generating sets for these ideals.  

One of the nice features of the Stratification of Goodearl and Lezter is that this stratification gives an efficient criterion to recognise primitive ideals. Indeed, by the Stratification Theorem of Goodeal and Lezter, primitive ideals are are exactly those primes that are maximal within their strata.  So, in order to study the primitive spectrum of $U_q[w]$, it is crucial to understand 
the strata, and in particular the dimension of these spaces. This is especially important because, by analogy with the Nullstellensatz in the commutative setting, primitive ideals correspond to the points of the ``quantum Schubert cells''. 

In previous work \cite{BL}, the first and third named authors gave a non-efficient formula for the dimension of a stratum. 
More precisely, if $v,w \in W$ with $v \leq w$ in the Bruhat ordering, then we described a matrix $M(v,w)$ such that the dimension of the stratum associated to $v$ in $U_q[w]$ is actually equal to the dimension of the kernel of this matrix $M(v,w)$ that we describe now. 

If $w= s_{i_1} \circ \cdots \circ s_{i_t}$ $(i_j \in \{1, \dots , n\})$ is a reduced decomposition of $w$ and $\beta_1 = \alpha_{i_1}$, $\beta_2 = s_{i_1}(\alpha_{i_2})$, $\ldots$, $\beta_t = s_{i_1} \circ \cdots \circ s_{i_{t-1}}(\alpha_{i_t})$ 
are the distinct positive roots associated to this reduced decomposition, then the matrix $M( \mathrm{id},w)$ is the following skew-symmetric matrix:
$$M( \mathrm{id},w)= \left( \begin{array}{ccccc}
0 & \langle \beta_1 , \beta_2 \rangle &  \dots & \dots &  \langle \beta_1 , \beta_t \rangle \\
 - \langle \beta_1 , \beta_2 \rangle & 0 & \ddots & &  \vdots\\
 \vdots & \ddots & \ddots& \ddots & \vdots\\
 \vdots &  &\ddots &0&  \langle \beta_{t-1} , \beta_t \rangle\\
  -\langle \beta_1 , \beta_t \rangle & \dots & \dots & - \langle \beta_{t-1} , \beta_t \rangle & 0\\
\end{array}
\right),$$
and in general $M(v,w)$ is a square submatrix of $M(\mathrm{id},w)$---this matrix will be described more precisely later in Section \ref{primeUqw}. Hence, in type $A_n$, the size of $M(\mathrm{id},w)$ can be as large as ${n \choose 2 } \times  {n \choose 2}$, whereas the dimension of any stratum is less than or equal to $n$. This example shows how ineffective the formula established in \cite{BL} is. 

The aim of this paper is to prove an efficient formula for the dimension of a stratum. Roughly speaking, our main theorem can be formulated as follows:

\begin{thm}
Let $v \leq w$. The dimension of the stratum associated to $v$ in $U_q[w]$ is equal to the dimension of $\ker (v+w)$. 
\end{thm}

Our proof is actually constructive as we exhibit explicit inverse isomorphisms between the kernel of the matrix $M(v,w)$ and 
$\ker (v+w)$. These isomorphisms, and the Lie theoretic formulae established to construct them, are of independent interest. 
The above theorem follows easily from these isomorphisms as the dimension of the stratum associated to $v$ is equal to the dimension of the kernel of $M(v,w)$ by \cite{BL}. 

Note that this result was established in the case where $v = \mathrm{id}$ in \cite{BL} thanks to earlier results of De Concini, Kac and Procesi.  

The above theorem has been independently established by Yakimov \cite{Y2} with slightly stronger hypotheses on the base field $\mathbb{K}$ and on the parameter $q$ by using representation theoretic techniques.

Quantum matrices are special examples of quantum Schubert cells. Even in this case, our formula for the dimension of a stratum is new. In a subsequent paper \cite{BCL1}, 
we use the present results to give enumeration results regarding primitive ideals in quantum matrices. This generalises previous results of the first and third named authors \cite{BLL,BLN}.\\

From the point of view of noncommutative algebraic geometry, the algebras $U_q[w]$ can be viewed as quantisations of algebras of regular functions on Schubert cells. 
In other words, these algebras should be thought of as algebras of regular functions on ``quantum Schubert cells''---some unknown spaces. The above theorem allows us to describe, set-theoretically at least, these varieties over an algebraically closed field. In this case, if we think of primitive ideals as the points of these varieties (by analogy with the Nullstellensatz), the ``quantum Schubert cell'' associated to $w \in W$ is the following set:
$$\bigsqcup_{v\leq w} (\mathbb{K}^*)^{\dim (\ker(v+w))}.$$



Throughout this paper, we use the following conventions. 
\begin{enumerate}
\item[$(i)$]
 If $I$ is a finite set, $|I|$ denotes its cardinality.  
\item[$(ii)$]  $\gc a,b \dc := \{ i\in{\mathbb N} \mid a\leq i\leq b\}$ for all integers $a$ and $b$, and for any natural number $t$, we set $[t]:=\{1, \dots , t\}=\gc 1,t \dc$. 
\item[$(iii)$] $\mathbb{K}$ denotes a field and we set
$\mathbb{K}^*:=\mathbb{K}\setminus \{0\}$. 
\item[$(iv)$] $q \in \C^*$ is not a root of unity. 
\item[$(v)$] If $A$ is a $\mathbb{K}$-algebra, then $\spec(A)$ and  $\prim(A)$ denote respectively its prime and primitive spectra.
\end{enumerate}

\section{Prime and primitive spectra of $U_q[w]$}

\subsection{Notation}

Let $\g$ be a simple Lie $\comp$-algebra of rank $n$. We denote by $\pi=\{\alpha_1,\dots,\alpha_n\}$ the set of simple roots
associated to a triangular decomposition $\g=\mathfrak{n}^- \oplus \mathfrak{h} \oplus \mathfrak{n}^+$. Recall that $\pi$ is a basis of a Euclidean vector space $E$ over $\mathbb{R}$, whose inner product is denoted by $\langle \mbox{ },\mbox{ } \rangle$ ($E$ is usually denoted by $ \mathfrak{h}_{\mathbb{R}}^*$ in Bourbaki). Recall that the scalar product of two roots $\langle \alpha,\beta \rangle$ is always an integer. We assume that the short roots have length $\sqrt{2}$, and we denote by $Q$ the root lattice of $\g$, that is:
 $$Q:=\bigoplus_{i=1}^n \B{Z}\alpha_i.$$

We denote by  $W$ the Weyl group of $\g$, that is, the subgroup of the orthogonal group of $E$ generated by the reflections $s_i:=s_{\alpha_i}$, for $i \in \{1,\dots,n\}$, with reflecting hyperplanes $H_i:=\{\beta \in E \mid (\beta,\alpha_i)=0\}$, $i \in \{1,\dots,n\}$. The length of $w \in W$ is denoted  by $l(w)$.  

We denote by $A=(a_{ij}) \in M_n(\mathbb{Z})$ the Cartan matrix associated to these data. As $\g$ is simple, $a_{ij} \in \{0,-1,-2,-3\}$ for all $i \neq j$. 
 
 If $\alpha$ is a root of $\g$, then we denote by $\alpha^\vee$ its associated co-root; recall that $$\alpha^\vee := \frac{2}{\langle\alpha\mid\alpha\rangle}\alpha.$$ 
We denote by $\varpi_1, \dots , \varpi_n$ the fundamental weights of $\g$; recall that these are the elements in the dual space of $E$ defined by 
$$\langle\varpi_j\mid\alpha_i^\vee\rangle= \delta_{i,j}.$$  
We denote by $P$ the associated weight lattice, that is 
$$P:=\bigoplus_{i=1}^n \B{Z}\varpi_i.$$ 
Recall  that for all $j$, we may write $$\alpha_j= \sum_{i}a_{i,j}\varpi_i.$$

\subsection{Quantised enveloping algebras}

For all $i \in \{1,\dots,n \}$, set
$q_i:=q^{\frac{(\alpha_i,\alpha_i)}{2}}$ and 
$$\left[ \begin{array}{l} m \\ k \end{array} \right]_i:=
\frac{(q_i-q_i^{-1}) \dots
  (q_i^{m-1}-q_i^{1-m})(q_i^m-q_i^{-m})}{(q_i-q_i^{-1})\dots
  (q_i^k-q_i^{-k})(q_i-q_i^{-1})\dots (q_i^{m-k}-q_i^{k-m})} $$
for all integers $0 \leq  k \leq  m$. By convention, 
$$\left[ \begin{array}{l} m \\ 0 \end{array} \right]_i:=1.$$

The quantised enveloping algebra $U_q(\g)$ of $\g$ over $\comp$ associated to
the previous data is the $\mathbb{K}$-algebra generated by the
indeterminates $E_1,\dots,E_n,F_1,\dots , F_n,K_1^{\pm 1}, \dots, K_n^{\pm 1}$ subject to the following relations:
$$K_i K_j =K_j K_i $$
$$ K_i E_j K_i^{-1}=q_i^{a_{ij}}E_j  \mbox{ and }  K_i F_j K_i^{-1}=q_i^{-a_{ij}}F_j$$
$$E_i F_j -F_jE_i=\delta_{ij} \frac{K_i-K_i^{-1}}{q_i-q_i^{-1}} $$
and the quantum Serre relations:
\begin{eqnarray}
\label{Serrequantique} 
\sum_{k=0}^{1-a_{ij}} (-1)^k  \left[ \begin{array}{c} 1-a_{ij} \\ k
 \end{array} \right]_i E_i^{1-a_{ij} -k} E_j E_i^k=0  \mbox{ } (i \neq  j)
\end{eqnarray}
and 
$$\sum_{k=0}^{1-a_{ij}} (-1)^k  \left[ \begin{array}{c} 1-a_{ij} \\ k
 \end{array} \right]_i F_i^{1-a_{ij} -k} F_j F_i^k=0  \mbox{ } (i \neq  j).$$

We refer the reader to \cite{bgbook,jantzen,josephbook} for more details on
this (Hopf) algebra. Further, as usual, we denote by $U_q^+(\g)$ (resp. $U_q^-(\g)$)  the
subalgebra of $U_q(\g)$ generated by $E_1,\dots,E_n$
(resp. $F_1,\dots,F_n$) and by $U^0$ the subalgebra of $U_q(\g)$ generated by 
$K_1^{\pm 1},\dots, K_n^{\pm 1}$.

\subsection{Braid group action and quantum Schubert cells}

To each reduced decomposition of the longest element of the Weyl group $W$ of $\g$, Lusztig has associated a PBW basis of $\U$, see for instance \cite[Chapter 37]{lusztigbook}, \cite[Chapter 8]{jantzen} or \cite[I.6.7]{bgbook}. The construction relates to a braid group action by automorphisms on $\U$. We use the convention of \cite[Chapter 8]{jantzen}. In particular, for any $\alpha \in \pi$, we define the braid automorphism $T_{\alpha}$ of the algebra $U_q(\g)$ as in 
\cite[p. 153]{jantzen}. We set $T_i:=T_{\alpha_i}$. It was proved by Lusztig that the automorphisms $T_{i}$ satisfy the braid relations, that is, if $s_is_j$ has order $m$ in $W$, then $$T_iT_jT_i \dots = T_j T_i T_j \dots ,$$
where there are exactly $m$ factors on each side of this equality.

Consider any $w\in W$, and set $t := l(w)$. Let $w= s_{i_1} \circ \cdots \circ s_{i_t}$ $(i_j \in \{1, \dots , n\})$ be a reduced decomposition of $w$. It is well known that
$\beta_1 = \alpha_{i_1}$, $\beta_2 = s_{i_1}(\alpha_{i_2})$, $\ldots$, $\beta_t = s_{i_1} \circ \cdots \circ s_{i_{t-1}}(\alpha_{i_t})$ 
are distinct positive roots and that the set $\{\beta_1, ..., \beta_t\}$ does not depend on the chosen reduced expression of $w$. 
Similarily, we define elements $E_{\beta_k}$ of $U_q(\g)$ by
$$E_{\beta_k}:= T_{i_1} \cdots T_{i_{k-1}} (E_{i_k}).$$
Note that the elements $E_{\beta_k}$ depend on the reduced decomposition of $w$. The following well-known results were proved by Lusztig and Levendorskii-Soibelman.

\begin{thm}[See for instance \cite{LevSoi}]
\label{theofond}
$ $
\begin{enumerate}
 \item For all $k \in \{ 1, \dots, t\}$, the element $E_{\beta_k}$ belongs to $\U$.
 \item If $\beta_k=\alpha_i$, then $E_{\beta_k}=E_i$.
 \item For all $1 \leq i < j \leq t$, we have 
 $$E_{\beta_j} E_{\beta_i} -q^{-(\beta_i , \beta_j)} E_{\beta_i} E_{\beta_j}=
 \sum a_{k_{i+1},\dots,k_{j-1}} E_{\beta_{i+1}}^{k_{i+1}} \cdots E_{\beta_{j-1}}^{k_{j-1}},$$
 where each $a_{k_{i+1},\dots,k_{j-1}} $ belongs to $\C$.
\end{enumerate}
\end{thm}

We denote by $U_q[w]$ the subalgebra of $\U$ generated by $E_{\beta_1}, \dots, E_{\beta_t}$.  It is well known that $U_q[w]$ does not depend on the reduced decomposition of $w$ (see \cite{DKP}). Moreover, the monomials $E_{\beta_1}^{k_1} \cdots E_{\beta_t}^{k_t}$, with $k_1, \dots, k_t \in \mathbb{N}$, form a linear basis of $U_q[w]$ (see \cite{DKP}).
As a consequence of this result, $U_q[w]$ can be presented as an iterated Ore extension over $\C$: 
$$U_q[w]= \C [E_{\beta_1}] [E_{\beta_2}; \sigma_2 , \delta_2] \cdots [E_{\beta_t}; \sigma_t , \delta_t],$$
 where each $\sigma_i$ is a linear automorphism and each $\delta_i$ is a $\sigma_i$-derivation of the appropriate 
 subalgebra. In other words, $U_q[w]$ is a skew polynomial ring whose multiplication is defined by:
 $$E_{\beta_j} a = \sigma_j(a) E_{\beta_j} + \delta_j(a)$$
 for all $j\in \gc 2 ,t \dc$ and $a \in \C [E_{\beta_1}] [E_{\beta_2}; \sigma_2 , \delta_2] \cdots [E_{\beta_{j-1}}; \sigma_{j-1} , \delta_{j-1}]$. Note that we don't have an explicit formula for $\delta_j$ in general. However, we have one for $\sigma_j$:
 $$\sigma_j(E_{\beta_i})=q^{-(\beta_i , \beta_j)} E_{\beta_i} $$
 for all $i<j$.
 
 As a corollary of this presentation as a skew-polynomial algebra, $U_q[w]$ is a noetherian domain and its group of invertible elements is reduced to nonzero elements of the base-field.


 \subsection{Torus action on quantum Schubert cells}
 
It is well known that the torus $\hc:=(\mathbb{K}^*)^n$ acts rationally by automorphisms on $\U$ via:
$$(h_1,\dots, h_n).E_i = h_i E_i \mbox{ for all } i \in \{1,\dots,n\}.$$
(It is easy to check that the quantum Serre relations are preserved by the group $\hc$.) It is also well known that this action of $\hc$ on $\U$ restricts to a rational action of $\hc$ on $U_q[w]$.

\subsection{Prime and primitive spectra}
\label{primeUqw}

Recall that if $A$ is a $\mathbb{K}$-algebra, then $\spec(A)$ and  $\prim(A)$ denote respectively its prime and primitive spectra. First, as $q$ is not a root of unity, it follows from  \cite{gletpams} that all prime ideals in $U_q[w]$ are completely prime (see also \cite[Th\'eor\`eme 6.2.1]{c1}).

We follow the general strategy of Goodearl and Letzter to study the prime and primitive spectra of $U_q[w]$; this strategy relies very much on the rational action of the torus $\hc$. More precisely, let $\hc$-$\spec(U_q[w])$ be the set of those prime ideals that are invariant under $\hc$. For each $J \in  \hc$-$\spec(U_q[w])$, let $\spec_J(U_q[w])$ be the so-called $\hc$-stratum associated to $J$, that is:
$$\spec_J(U_q[w]) = \{ P \in \spec(U_q[w])  \mid  \bigcap_{h \in \hc} h.P=J\}.$$
Then it follows from the Stratification Theorem of Goodearl and Letzter (see for instance \cite[II.2]{bgbook}) that the $\hc$-strata form a partition of the prime spectrum of $U_q[w]$:

$$\spec(U_q[w]) =\bigsqcup_{J \in  \hc \mbox{-}\spec(U_q[w])} \spec_J(U_q[w]).$$

Recently and independently, $\hc$-primes in $U_q[w]$ have been studied by M\'eriaux-Cauchon and 
Yakimov. M\'eriaux and Cauchon \cite{CM}  proved that $\hc$-$\spec(U_q[w])$ is in one-to-one correspondence with 
the initial Bruhat interval $\{v\in W  \mid  v \leq w\}$. Yakimov \cite{Y} subsequently proved that $\hc$-$\spec(U_q[w])$ and $\{v\in W  \mid  v \leq w\}$ are isomorphic as posets, and provides generating sets for all $\hc$-primes. Let us point out that the parametrisations obtained by M\'eriaux-Cauchon on one hand, and Yakimov on the other hand, are obtained via very different methods and to the best of our knowledge it is not known whether these two parametrisations coincide.

From the point of view of representation theory, it is important to be able to describe the primitive ideals of the quantum Schubert cell $U_q[w]$. To achieve this, the Stratification Theorem is a powerful tool; indeed, it follows from \cite[Theorem II.8.4]{bgbook} that the primitive ideals are exactly those primes that are maximal within their $\hc$-strata. Of course, this provides an easy criterion to recognise primitive ideals. Moreover the geometric nature of the $\hc$-strata is known; each $\hc$-stratum is homeomorphic to the prime spectrum of a commutative Laurent polynomial ring over $\C$. One crucial piece of information that is missing is in fact the dimension of this commutative Laurent polynomial ring. 

In \cite{BL}, a formula was established for this dimension; it was expressed as the dimension of the kernel of a certain matrix. 
Before giving more details, let us point out that this formula is not really effective as for instance in type $A_n$ this matrix can be of size ${ n\choose 2} \times { n\choose 2} $ whereas it was observed that the dimension of the kernel never exceeds $n$.
  
Let us now recall the formula that has been obtained in \cite{BL}. This formula relies on the parametrisation of the $\hc$-prime spectrum of $U_q[w]$ obtained by M\'eriaux-Cauchon. Recall that $w= s_{i_1} \circ \cdots \circ s_{i_t}$ $(i_{\ell} \in \{1, \dots , n\})$ is a fixed reduced decomposition of $w\in W$. In \cite{c1}, Cauchon used his deleting-derivations algorithm in order to construct an embedding $\varphi$ from  $\hc$-$\spec(U_q[w])$ into the power set $\mathcal{P}_t$ of $[t]$ of all subsets of $[t]$. We call any subset of $[t]$ a {\it diagram}. Before describing the image of this embedding we need to introduce some notation.

\begin{notn} {\rm 
Fix a reduced decomposition $w= s_{i_1} \circ \cdots \circ s_{i_t}$ $(i_j \in \{1, \dots , n\})$ of $w$ and a diagram $\Delta\subseteq\{1,2,\ldots,t\}$. 
\begin{enumerate}
\item For all $k\in [t]$, we set \[s_{i_k}^{\Delta}:=\left\{ \begin{array}{ll} s_{i_k} & {\rm if~}k\in \Delta, \\{\rm id} & {\rm otherwise.}\end{array}\right. \]
\item We set $\{t_1< \ldots < t_m\}:=\{1, \dots, l(w)\} \setminus\Delta$ and 
$j_{\ell}:=i_{t_{\ell}}$ for all $\ell \in [m]$.
\item We also set  $$w^\Delta:=s_{i_1}^\Delta\cdots s_{i_t}^\Delta\in W,$$
$$v^\Delta:=s_{i_t}^\Delta\cdots s_{i_1}^\Delta\in W$$
and 
$$v_{\ell}^\Delta:=s_{i_t}^\Delta\cdots s_{i_{t-\ell+1}}^\Delta\in W$$
for all $\ell \in \{0, \dots , t\}$.
\end{enumerate}}
\end{notn}

\begin{defn}
{\rm A diagram $\Delta$ is {\it Cauchon (with respect to the fixed reduced decomposition of $w$)} if 
$$v_{\ell-1}^{\Delta} < v_{\ell-1}^{\Delta} s_{i_{t-\ell+1}} $$
for all  $\ell \in \{1, \dots , t\}$.}
\end{defn}

In \cite{CM}, Cauchon diagrams are called admissible diagrams. It was proved by M\'eriaux and Cauchon \cite{CM} that they coincide 
with positive subexpressions of the reduced decomposition of $w$ in the sense of Marsh and Rietsch \cite{MR}. 

Recently, the image of this embedding $\varphi: \hc \mbox{-}\spec(U_q[w]) \rightarrow \mathcal{P}_t$ was described by M\'eriaux and Cauchon; they proved that $\varphi$ induces a one-to-one correspondence between $\hc$-$\spec(U_q[w])$ and the set of Cauchon diagrams associated to the chosen  reduced decomposition of $w$ \cite[Theorem 5.3.1]{CM}. 

As their bijection is actually obtained through the deleting-derivations theory, the results obtained in \cite{BL} can be applied. In particular, with the above notation, we deduce from \cite[Theorem 3.1 and Section 3C]{BL} the following formula for the dimension of an $\hc$-stratum.

\begin{thm}[\cite{BL}]
\label{theoBL}
Let $J$ be the unique $\hc$-prime of $U_q[w]$ associated to the Cauchon diagram $\Delta$. Let $M(w^\Delta)$ be the $m\times m$ skew-symmetric matrix defined by setting the $(i,\ell)$-entry $(i<\ell)$ to be $\langle \beta_{j_i}|\beta_{j_{\ell}}\rangle$. Then, with the above notation, the following holds:
$$\dim (\spec_J(U_q[w]) ) = \dim (\ker(M(w^\Delta))).$$
\end{thm}
Observe that when $w$ is the longest word in type $A_n$ and $\Delta$ is the empty diagram (which is easily seen to be a Cauchon diagram) then  $M(w^\Delta)$ is a matrix of size ${ n\choose 2} \times { n\choose 2} $. However it was proved by De Concini and Procesi that the kernel of this matrix is actually isomorphic to the kernel of an $n \times n$ matrix \cite[Lemma 10.4 and 10.6]{DP}. This is their result that we will generalise in the next section.

\section{Dimension of the strata}

In this section we fix a reduced decomposition $$w=s_{i_1}s_{i_2}\cdots s_{i_t}$$ of $w \in W$. Our aim is to compute the dimension of the kernel of the matrix $M(w^\Delta)$ introduced in Theorem \ref{theoBL} for any diagram $\Delta$.

\subsection{Notation and main theorem}
\label{not1}

We start by fixing the notation.

\begin{notn} {\rm Throughout the next sections, we set the following notation.
\begin{enumerate}
\item Fix $\Delta\subseteq\{1,2,\ldots,t\}$,  set $\{t_1< \cdots < t_m\}=\{1, \dots, l(w)\} \setminus\Delta$ and 
$j_{\ell}:=i_{t_{\ell}}$ for all $\ell \in [m]$.
\item For all $k\in [t]$, we set \[s_{i_k}^{\Delta}:=\left\{ \begin{array}{ll} s_{i_k} & {\rm if~}k\in \Delta, \\
{\rm id} & {\rm otherwise.}\end{array}\right. \]
\item $w^\Delta:=s_{i_1}^\Delta\cdots s_{i_t}^\Delta\in W$.
\item $\beta_1^\Delta := \alpha_{i_1}$ and for $k>1$, $$\beta_k^\Delta = s_{i_1}^\Delta\cdots s_{i_{k-1}}^\Delta(\alpha_{i_k}).$$ 
\end{enumerate}}
\end{notn}

Note that we do not need to assume that $\Delta $ is a Cauchon diagram.

Let $M(w^\Delta)$ be the $m\times m$ skew-symmetric matrix defined by setting the $(i,\ell)$-entry $(i<\ell)$ to be
$\langle \beta_{j_i}|\beta_{j_{\ell}}\rangle$; this is the matrix introduced in Theorem \ref{theoBL}.  We regard $M(w^\Delta)$ as a map from $\mathbb{Q}^m$ to itself.  Recall that our aim is to compute the kernel of this matrix. 

Set $V:=\sum_{i=1}^n \mathbb{Q}\cdot \varpi_i \ = \ P\otimes_{\mathbb Z} \mathbb{Q}$, and consider the map $w+w^\Delta:V \rightarrow V$. Our main result can be stated as follows.

\begin{thm} \label{blmain}
For any word $w$ in the Weyl group and any diagram $w^\Delta$ of $w$, $$\ker(M(w^\Delta))\simeq \ker(w+w^\Delta).$$
\end{thm}

The next section consists of some technical lemmas used in the construction of an explicit isomorphism needed to prove Theorem~\ref{blmain}. We note that De Concini and Procesi~\cite[Chapter 10]{DP} have proved the special case when $\Delta=\emptyset$, i.e., $w^\Delta=\textnormal{id}$, the identity map. 

\subsection{Two technical lemmas}
\label{not2}

Keeping in mind the notation introduced in the previous section, notice that there exist (possibly empty) words 
$w_0,\ldots ,w_m\in W$ such that 
$$w^{\Delta}=w_0w_1\cdots w_m$$ and
$$w= w_0s_{j_1}w_1s_{j_2}w_2\cdots s_{j_m}w_m,$$ where we write $j_{\ell}:=i_{t_{\ell}}$.  

\begin{notn}{\rm
For $w=w_0s_{j_1}w_1s_{j_2}w_2\cdots s_{j_m}w_m$ we fix the following notation.
\begin{enumerate}
\item For all $1\le \ell \le m$,
$$u_{\ell}^{\Delta}:=w_0s_{j_1}w_1s_{j_2}\cdots w_{\ell-2}s_{j_{\ell -1}} w_{\ell-1}(s_{j_{\ell}}-{\rm id})w_{\ell}w_{\ell+1}\cdots w_m.$$
\item For all $1\le \ell \le m$,
$$v_{\ell}^{\Delta}:=w_0s_{j_1}w_1s_{j_2}\cdots w_{\ell-2}s_{j_{\ell -1}} w_{\ell-1}(s_{j_{\ell}}+{\rm id})w_{\ell}w_{\ell+1}\cdots w_m.$$
\item In order to simplify notation, we write $\beta_{j_\ell}$ instead of  $\beta_{i_{t_{\ell}}}$, and $\beta_{j_\ell}^\Delta$ instead of $\beta_{i_{t_{\ell}}}^\Delta$ for each $\ell$.
\end{enumerate}}
\end{notn}

\begin{lem} \label{lem: id}
For all $\ell\in[m]$, we have
$$w+w^{\Delta} \ = \ v_{\ell}^{\Delta} - \sum_{h<\ell} u_h^{\Delta} + \sum_{h>\ell} u_h^{\Delta}.$$
\end{lem}
\begin{proof} 
Note that
\begin{eqnarray*} &~& \sum_{h<\ell} u_h^{\Delta} \\ &=&
 \sum_{h<\ell}  \left( w_0s_{j_1}w_1\cdots w_{h-1} s_{j_h} w_h\cdots w_m - w_0s_{j_1}w_1\cdots w_{h-2} s_{j_{h-1}} w_{h-1} w_h\cdots w_m\right) \\
 &=& w_0s_{j_1}\cdots w_{\ell-2}s_{j_{\ell-1}}w_{\ell-1}w_{\ell}\cdots w_m - w_0\cdots w_m \\
 &=&  w_0s_{j_1}\cdots w_{\ell-2}s_{j_{\ell-1}}w_{\ell-1}w_{\ell}\cdots w_m-w^{\Delta}.
 \end{eqnarray*}
 Similarly,
 $$ \sum_{h>\ell} u_h^{\Delta} = w-w_0s_{j_1}\cdots w_{\ell-1}s_{j_{\ell}}w_{\ell}w_{\ell+1}\cdots w_m.$$
 Thus
 \begin{eqnarray*} &~&  \sum_{h>\ell} u_h^{\Delta} - \sum_{h<\ell} u_h^{\Delta} \\
&=& w+w^{\Delta} - w_0s_{j_1}\cdots w_{\ell-1}s_{j_{\ell}}w_{\ell}w_{\ell+1}\cdots w_m  
-  w_0s_{j_1}\cdots w_{\ell-2}s_{j_{\ell-1}}w_{\ell-1}w_{\ell}\cdots w_m \\
&=& w+w^{\Delta}  - v_{\ell}^{\Delta}.
\end{eqnarray*}
The result follows.
\end{proof}

Recall that $V=\sum_{i=1}^n \mathbb{Q}\cdot \varpi_i \ = \ P\otimes_{\mathbb Z} \mathbb{Q}$.  
\begin{lem} If $\gamma\in V$, then for all $\ell\in[m]$ we have
$$u_{\ell}^{\Delta}(\gamma) =- \langle (\beta_{j_{\ell}}^{\Delta})^{\vee}|w^{\Delta}(\gamma)\rangle \beta_{j_{\ell}}$$
and
$$ \langle v_{\ell}^{\Delta}(\gamma)|\beta_{j_{\ell}}\rangle = 0.$$
\label{lem: rangle}
\end{lem}
\begin{proof}  Recall that
 $$u_{\ell}^{\Delta}=w_0s_{j_1}w_1s_{j_2}\cdots w_{\ell-2}s_{j_{\ell -1}} w_{\ell-1}(s_{j_{\ell}}-{\rm id})w_{\ell}w_{\ell+1}\cdots w_m.$$
Let $$\gamma' = w_{\ell}w_{\ell+1}\cdots w_m(\gamma).$$
Notice that $s_{j_{\ell}}-{\rm id}$ has kernel spanned by $\{\varpi_i\mid i\neq j_{\ell}\}$. Also, notice that since $\langle \alpha_j^\vee\mid \gamma^\prime\rangle$ is equal to the coefficient of $\varpi_{j}$ in $\gamma^\prime$, we may write $$\gamma' = \langle \alpha_{j_{\ell}}^{\vee} | \gamma'\rangle \varpi_{j_{\ell}} + \delta,$$ with
$\delta$ in the kernel of $s_{j_{\ell}}-{\rm id}$.  Since $s_{j_{\ell}}-{\rm id}$ sends $\varpi_{j_{\ell}}$ to $-\alpha_{j_{\ell}}$, we see that
$$u_{\ell}^{\Delta}(\gamma) =-  \langle \alpha_{j_{\ell}}^{\vee} | \gamma'\rangle w_0s_{j_1}w_1s_{j_2}\cdots w_{\ell-2}s_{j_{\ell -1}} w_{\ell-1}(\alpha_{j_{\ell}}),$$
which is just $$-  \langle \alpha_{j_{\ell}}^{\vee} | \gamma'\rangle \beta_{j_{\ell}}.$$
Since the operators $s_{i_k}$ are isometries of $V$ and its dual space, we see that
\begin{eqnarray*} &~&  \langle \alpha_{j_{\ell}}^{\vee} | \gamma'\rangle \\ &=&
 \langle w_0\cdots w_{\ell-1}(\alpha_{j_{\ell}}^{\vee}) | w_0\cdots w_{\ell-1}(\gamma')\rangle \\
 &=&  \langle (\beta_{j_{\ell}}^{\Delta})^{\vee}|w_0\cdots w_m(\gamma)\rangle \\
 &=& \langle (\beta_{j_{\ell}}^{\Delta})^{\vee}|w^{\Delta}(\gamma)\rangle.
 \end{eqnarray*}
 Thus
 $$u_{\ell}^{\Delta}(\gamma)=- \langle (\beta_{j_{\ell}}^{\Delta})^{\vee}|w^{\Delta}(\gamma)\rangle \beta_{j_{\ell}}.$$
 
Observe that
$$ v_{\ell}^{\Delta}(\gamma)\in w_0s_{j_1}w_1s_{j_2}\cdots w_{\ell-2}s_{j_{\ell -1}} w_{\ell-1}(s_{j_{\ell}}+{\rm id})(V).$$

On the other hand, for all $k \in [n]$, we have 
\begin{eqnarray*}
\langle(s_{j_{\ell}}+{\rm id})(\varpi_k)\mid \alpha_{j_l}\rangle & = &  \langle s_{j_l}(\varpi_k)\mid \alpha_{j_l}\rangle + \langle \varpi_k\mid \alpha_{j_l}\rangle \\
& = & \langle \varpi_k\mid s_{j_l}(\alpha_{j_l})\rangle + \langle \varpi_k\mid \alpha_{j_l}\rangle \\ 
& = & -\langle \varpi_k\mid \alpha_{j_l}\rangle + \langle \varpi_k\mid \alpha_{j_l}\rangle \\ 
& = & 0.
\end{eqnarray*}
In other words, $\frac{1}{2}(s_{j_{\ell}}+{\rm id})$ is simply a projection onto the orthogonal complement of $\mathbb{Q}\alpha_{j_{\ell}}$. It follows that
$$v_{\ell}^{\Delta}(\gamma)\in w_0s_{j_1}w_1s_{j_2}\cdots w_{\ell-2}s_{j_{\ell -1}} w_{\ell-1}(\langle \alpha_{j_{\ell}}\rangle^{\perp}).$$
In particular,
$v_{\ell}^{\Delta}(\gamma)$ is orthogonal to 
$w_0s_{j_1}w_1s_{j_2}\cdots w_{\ell-2}s_{j_{\ell -1}} w_{\ell-1}(\alpha_{j_{\ell}})=\beta_{j_{\ell}}$.
This completes the proof.
\end{proof}

Let $\vb{e}_{\ell}$ be the vector in $\mathbb{Q}^m$ whose $\ell^\textnormal{th}$ coordinate is $1$ and all other coordinates are zero. For a statement $X$, we let $\chi(X)=1$ if $X$ is true, and 0 otherwise. As a result of the preceding lemmas, we obtain the next result.
\begin{cor} 
For $1\le \ell\le m$ and $\gamma\in V$ we have
$$\langle (w+w^{\Delta})(\gamma)|\beta_{j_{\ell}}\rangle \ = \ \sum_{h\neq \ell} (-1)^{\chi(h>\ell)} \langle \beta_{j_h}|\beta_{j_{\ell}}\rangle \langle (\beta_{j_{h}}^{\Delta})^{\vee}|w^{\Delta}(\gamma)\rangle.$$
\label{cor: wdel}
\end{cor}
\begin{proof} Recall that Lemma~\ref{lem: id} gives
$$w+w^{\Delta} = v_{\ell}^{\Delta} + \sum_{h\neq \ell } (-1)^{\chi(h<\ell)} u_{h}^{\Delta}.$$
We apply both sides to $\gamma\in V$, and then apply the operator $\langle \,  \cdot \, | \beta_{j_{\ell}}\rangle$.    By Lemma \ref{lem: rangle}, this operator annihilates $v_{\ell}^{\Delta}(\gamma)$ and
sends $u_{h}^{\Delta}(\gamma)$ to 
$$- \langle (\beta_{j_{h}}^{\Delta})^{\vee}|w^{\Delta}(\gamma)\rangle \langle \beta_{j_h}|\beta_{j_{\ell}}
\rangle .$$ Thus 
$$\langle (w+w^{\Delta})(\gamma)|\beta_{j_{\ell}}\rangle = 
 \sum_{h\neq \ell} (-1)^{\chi(h>\ell)} \langle \beta_{j_h}|\beta_{j_{\ell}}\rangle \langle (\beta_{j_h}^{\Delta})^{\vee}|w^{\Delta}(\gamma)\rangle,$$ as required.
\end{proof}

\subsection{Proof of the main theorem}

We consider the map $B:V\to \mathbb{Q}^{m}$ defined by
$$B(\gamma)= \sum_{\ell=1}^m \langle\gamma|\beta_{j_{\ell}} \rangle \vb{e}_{\ell}.$$ 
The bulk of the work needed to prove Theorem~\ref{blmain} is contained in the following result.

\begin{prop}\label{thm: f}
Let $f:\mathbb{Q}^m\oplus V\to \mathbb{Q}^m$ be defined by
$$f(x,y):=\left(M(w^{\Delta})\right)(x)+B(y).$$
The function $f$ is surjective and a basis for the kernel is given by the elements of the form
$$\left(\sum_{h=1}^m \langle  (\beta_{j_{h}}^{\Delta})^{\vee} | w^{\Delta}(\varpi_j)\rangle  \vb{e}_h, (w+w^{\Delta})(\varpi_j)\right),$$ for $j\in [n]$.
\end{prop}
\begin{proof}
Note that
$$f( \vb{e}_h,\beta_{j_h})=\sum_{\ell\neq h} (-1)^{\chi(\ell>h)} \langle \beta_{j_\ell},\beta_{j_h}\rangle  \vb{e}_\ell 
+\sum_{\ell=1}^m \langle \beta_{j_\ell},\beta_{j_h}\rangle  \vb{e}_\ell 
=  \langle \beta_{j_h},\beta_{j_h}\rangle  \vb{e}_h + 2\sum_{\ell<h} \langle \beta_{j_\ell},\beta_{j_h}\rangle  \vb{e}_\ell .$$
Since $\alpha\neq 0$ for any root $\alpha$, these elements span $\mathbb{Q}^m$ as $h$ ranges from $1$ to $m$.  
Therefore, $f$ is onto. For this reason, we know that the dimension of the kernel of $f$ is $n$.  Thus to see that the elements
$$\left(\sum_{h=1}^m \langle  (\beta_{j_{h}}^{\Delta})^{\vee} | w^{\Delta}(\varpi_j)\rangle  \vb{e}_h, (w+w^{\Delta})(\varpi_j)\right),$$ for $j=1,2,\ldots ,n$, form a basis, it is enough to show that they are linearly independent and are in the kernel. 
Note that by Corollary \ref{cor: wdel} we have that for any $\gamma\in V$,
$$B((w+w^{\Delta})(\gamma))=\sum_{\ell=1}^m \left( \sum_{h\neq \ell} (-1)^{\chi(h>\ell)} \langle \beta_{j_h}|\beta_{j_{\ell}}\rangle \langle (\beta_{j_{h}}^{\Delta})^{\vee}|w^{\Delta}(\gamma)\rangle\right)  \vb{e}_\ell.$$
But this is just $$-M(w^{\Delta})\left( \sum_{h=1}^m \langle  (\beta_{j_{h}}^{\Delta})^{\vee} | w^{\Delta}(\gamma)\rangle  \vb{e}_h\right),$$ and so these elements are indeed in the kernel.  

Now for $a_i\in \B{Q}$, with $i\in[n]$, consider the linear combination 
\begin{eqnarray*}
&&\sum_{i=1}^na_i\left(\sum_{h=1}^m \langle  (\beta_{j_{h}}^{\Delta})^{\vee} | w^{\Delta}(\varpi_i)\rangle  \vb{e}_h, (w+w^{\Delta})(\varpi_i)\right) \\
&&=  \left(\sum_{h=1}^m \langle  (\beta_{j_{h}}^{\Delta})^{\vee} | w^{\Delta}\left(\sum_{i=1}^na_i\varpi_i\right)\rangle  \vb{e}_h, (w+w^{\Delta})\left(\sum_{i=1}^na_i\varpi_i\right)\right)\\
&&= \left(\sum_{h=1}^m \langle  (\beta_{j_{h}}^{\Delta})^{\vee} | w^{\Delta}(\gamma)\rangle  \vb{e}_h, (w+w^{\Delta})(\gamma)\right),
\end{eqnarray*}
where $\gamma:= a_1 \varpi_1 + \dots +a_n \varpi_n$. Therefore, to check that the kernel elements in the statement of the theorem are linearly independent, it suffices to show that if $\gamma\in {\rm ker}(w+w^{\Delta})$ and $\langle   (\beta_{j_{h}}^{\Delta})^{\vee} | w^{\Delta}(\gamma)\rangle = 0$ for all $h\in[m]$, then $\gamma$ must be zero.  

Suppose that $\gamma\in  {\rm ker}(w+w^{\Delta})$  satisfies
$$\langle (\beta_{j_h}^{\Delta})^{\vee}|w^{\Delta}(\gamma)\rangle = 0$$ for all $h\in[m]$.
Equivalently,
$$\langle w_0\cdots w_{h-1}(\alpha_{j_h}^{\vee})|w_0\cdots w_m(\gamma)\rangle = 0$$ for
 all $h\in[m]$.  
Since the $w_i$ are isometries, we see that $w_h\cdots w_m(\gamma)$ is annihilated by the linear functional $\alpha_{j_h}^{\vee}$ and hence
$$w_h\cdots w_m(\gamma) \in {\rm Span}\{\varpi_i\mid i \neq j_h\}=\ker(s_{j_h}-{\rm id}).$$
In particular, $s_{j_h}w_h\cdots w_m(\gamma)=w_h\cdots w_m(\gamma)$ for all $h\in[m]$.

To finish the proof, it suffices to prove that $w(\gamma)=w^{\Delta} (\gamma)$, as we have $w(\gamma)=-w^{\Delta}(\gamma)$ by assumption and  so linear independence will follow.
For $h\in[m]$ we let
$$\gamma_h^{\Delta}=w_h\cdots w_m(\gamma)$$ and
$$\gamma_h =s_{j_h}w_h s_{j_{h+1}} w_{h+1}\cdots s_{j_m} w_m(\gamma).$$  
We now show that for any $h\in[m]$, we have $\gamma_h^{\Delta}=\gamma_h.$

To this end, note that
$s_{j_m}w_m(\gamma)=w_m(\gamma)$, and so $\gamma_m=\gamma_m^{\Delta}$.  Suppose that there exists some $h$ such that $\gamma_h^{\Delta}\neq \gamma_h$. If $h$ is chosen to be maximal with respect to the property 
$\gamma_h^{\Delta}\neq \gamma_h$,
then 
\begin{eqnarray*}
\gamma_h^{\Delta} & = &  w_h\cdots w_m(\gamma)\\
&=& s_{j_h}w_h\cdots w_m(\gamma) \\
&=& s_{j_h}w_h(\gamma_{h+1}^{\Delta}) \\
&=& s_{j_h} w_h (\gamma_{h+1}) \\
&=& \gamma_h.\end{eqnarray*}
This is a contradiction.  Thus
$\gamma_h^{\Delta}=\gamma_h$ for $1\le h\le m$.  In particular, if we take $h=1$, then we see that
$$w^{\Delta}(\gamma)=w_0(\gamma_1^{\Delta})=w_0(\gamma_1)=w(\gamma),$$ which proves the proposition.

\end{proof}

We are now in position to prove our main result.

\begin{thm}
\label{theoBCL}
With the notation of Sections \ref{not1} and \ref{not2}, let $\Psi: {\rm ker}(w+w^{\Delta})\to {\rm ker}(M(w^{\Delta}))$ be defined by setting
$$\Psi(\gamma):=\sum_{h=1}^m \langle (\beta_{j_h}^{\Delta})^{\vee}|w^{\Delta}(\gamma)\rangle \vb{e}_h.$$
Then $\psi$ is an isomorphism.
\end{thm}

\begin{proof}
Note that if $\gamma$ is in the kernel of $w+w^{\Delta}$, then $\Psi(\gamma)$ is in ${\rm ker}(M(w^{\Delta}))$ by Proposition \ref{thm: f}.  Moreover, $\Psi$ is surjective, since if $x$ is in the kernel of $M(w^{\Delta})$, then $f(x,0)$ must be zero and by assumption $(x,0)$ is in the span of
the elements
$$\left(\sum_{h=1}^m \langle  (\beta_{j_{h}}^{\Delta})^{\vee} | w^{\Delta}(\varpi_j)\rangle \vb{e}_h, (w+w^{\Delta})(\varpi_j)\right),$$ for $j=1,2,\ldots ,m$, which means 
that $x=\Psi(\gamma)$ for some $\gamma$ in the kernel of $w+w^{\Delta}$.  We show that $\Psi$ is injective, which will complete the proof of Theorem~\ref{blmain}. 

For $\gamma\in  {\rm ker}(w+w^{\Delta})$ in the kernel of $\Psi$ we have
$$\langle (\beta_{j_h}^{\Delta})^{\vee}|w^{\Delta}(\gamma)\rangle = 0$$ for $1\le h\le m$. We showed, however, that this could not occur in the proof of the  linear independence of the kernel of the map $f$ in the proof of Proposition \ref{thm: f}.

\end{proof}

\subsection{Inverse bijection}

For completeness, we construct an inverse $\Phi$ to the $\Psi$ from the statement of Theorem \ref{theoBCL}.
Define $\Phi : {\mathbb Q}^m \to V$ by 
$$\Phi(\vb{e}_h) = \frac{1}{2}(w^{\Delta})^{-1}(\beta_{j_h})$$ for $h=1,\ldots ,m$. 

Note that since the map $\Psi$ is onto, if $x=\sum_{h=1}^m c_h \vb{e}_h\in {\rm ker}(M(w^{\Delta}))$, then there is some $\gamma$ in ${\rm ker}(w+w^{\Delta})$ such that 
$$c_h = \langle (\beta_{j_h}^{\Delta})^{\vee}|w^{\Delta}(\gamma)\rangle.$$

Now one can easily use the definitions to check that $u_1^\Delta=v_1^\Delta-2w^\Delta.$ 
But we know $(w+w^{\Delta})(\gamma)=0$. These facts, together with Lemma~\ref{lem: id}, give us
\begin{eqnarray*}
-2w^{\Delta}(\gamma) &=& (w+w^\Delta - 2w^{\Delta})(\gamma)\\
&=&(v_1^\Delta+\sum_{h>1}u_h^\Delta - 2w^\Delta)(\gamma)\\
&=& \sum_{h=1}^m u_{h}^{\Delta}(\gamma) \\
&=& -\sum_{h=1}^m \langle (\beta_{j_h}^{\Delta})^{\vee}|w^{\Delta}(\gamma)\rangle \beta_{j_h}\\
&=& -\sum_{h=1}^m c_h \beta_{j_h}.
\end{eqnarray*}
Thus 
$$\gamma=\frac{1}{2}\left((w^{\Delta})^{-1}\left(\sum_{h=1}^m c_h \beta_{j_h}\right)\right)=\Phi(x).$$
Thus we have shown that $\Phi$ induces a linear map $\tilde{\Phi} : {\rm ker}(M(w^{\Delta})) \to {\rm ker}(w+w^{\Delta})$ 
such that 
$$\Psi\circ \tilde{\Phi} = {\rm id}|_{{\rm ker}(w+w^{\Delta})},$$
and so we immediately obtain the following theorem.
\begin{thm} \label{inverses}
For the functions $\Psi$ and $\tilde{\Phi}$ defined above, we have 
$$\Psi\circ \tilde{\Phi}={\rm id}|_{{\rm ker}(w+w^{\Delta})} \mbox{ and }
\tilde{\Phi}\circ \Psi={\rm id}|_{{\rm ker}(M(w^{\Delta}))}.$$
\end{thm}

\subsection{An example}

Suppose that $\F{g}$ is of type $A_4$ with $w=s_1s_2s_1s_3s_2s_4s_3s_2$.   We have
\begin{eqnarray*}
\beta_1&=&\alpha_1;\\
\beta_2&=&\alpha_1+\alpha_2;\\
\beta_3&=&\alpha_2;\\
\beta_4&=&\alpha_1+\alpha_2+\alpha_3;\\
\beta_6&=&\alpha_2+\alpha_3;\\
\beta_6&=&\alpha_1+\alpha_2+\alpha_3+\alpha_4;\\
\beta_7&=&\alpha_2+\alpha_3+\alpha_4;\\
\beta_8&=&\alpha_4.
\end{eqnarray*}

Let $\Delta=\{2\}$. One can easily check that $\Delta$ is a Cauchon diagram. In this case, one has $w^\Delta = s_2$. We may write $$w=w_0s_1w_1s_1s_3s_2s_4s_3s_2w_2$$ where $w_0=w_2=\mathrm{id}$ and $w_1=s_2$. Moreover, one can check that  

$$M(w^\Delta)=\begin{bmatrix} 0& -1& 1& -1& 1& -1& 0 \\ 1&0&0&1&0&1&0 \\ -1&0&0&1&1&0&-1\\ 1&-1&-1&0&0&1&-1\\ -1&0&-1&0&0&1&1\\1&-1&0&-1&-1&0&1 \\ 0&0&1&1&-1&-1&0\end{bmatrix}.$$ 

On the other hand, it is easy to compute the $\beta_i^\Delta$:
\begin{eqnarray*}
\beta_1^\Delta&=&\alpha_1;\\
\beta_2^\Delta&=&\alpha_2;\\
\beta_3^\Delta&=&\alpha_1+\alpha_2;\\
\beta_4^\Delta&=&\alpha_2+\alpha_3;\\
\beta_5^\Delta&=&-\alpha_2;\\
\beta_6^\Delta&=&\alpha_4;\\
\beta_7^\Delta&=&\alpha_2 + \alpha_3;\\
\beta_8^\Delta&=&-\alpha_2;\\
\end{eqnarray*}
So the matrix of $w+w^\Delta$ is the basis of simple roots is given by
$$w+w^\Delta= \begin{bmatrix} 1 & 0 & 0 & -1 \\
1 & -1 & 0 & 0 \\
1 & 0 & 0 & 0\\
1 & -1 & 0 & 1\\
\end{bmatrix}$$
The kernel of this linear map is spanned by $\alpha_3$. One can easily check that:
$$\psi(\alpha_3)=-\vb{e}_1+2\vb{e}_3-\vb{e}_4-\vb{e}_5+2\vb{e}_6-\vb{e}_7, $$ 
and that this vector actually spans the kernel of the matrix $M(w^\Delta)$.

\subsection{Dimension of $\hc$-strata in $U_q[w]$}

As a consequence of Theorems \ref{theoBL} and \ref{theoBCL}, we obtain the following formula for the dimension 
of the $\hc$-stratum associated to the $\hc$-prime attached to the Cauchon diagram $\Delta$.

\begin{thm}
Let $J$ be the unique $\hc$-prime of $U_q[w]$ associated to the Cauchon diagram $\Delta$. Then, with the notation of the previous sections, we have:
$$\dim (\spec_J(U_q[w]) ) = \dim (\ker(w^\Delta +w)).$$
\end{thm}

As a consequence of this result we obtain that the dimension of an $\hc$-stratum is always less than or equal to the rank of the underlying simple Lie algebra. 

A similar result has been obtained independently, and with completely different methods, by Yakimov \cite{Y2} under the assumption that $\mathbb{K}$ is a field of characteristic $0$ and that $q$ is transcendental over $\mathbb{Q}$.

When $\g$ is of type $A_n$ and $w$ is a certain element of $W$, the algebra $U_q[w]$ is isomorphic to the algebra of $p\times m$ quantum matrices. We refer the interested reader to \cite[Section 2.1]{CM} for more details. The results obtained here, and in particular Theorem \ref{theoBCL} will be used in a subsequent paper \cite{BCL1} in order to complete the current project \cite{BLL,BLN} of enumerating primitive $\hc$-primes in quantum matrices.

\section*{Acknowledgements}

We would like to thank Milen Yakimov for telling us about the results in \cite{Y2}.

\newpage

\noindent Jason Bell:\\
Department of Mathematics\\
Simon Fraser University\\
Burnaby, BC V5A 1S6, Canada\\
Email: {\tt jpb@math.sfu.ca}\\

\noindent Karel Casteels:\\
Department of Mathematics\\
University of California\\
Santa Barbara, CA 93106, USA\\
Email: {\tt casteels@math.ucsb.edu}\\

\noindent St\'ephane Launois \\
School of Mathematics, Statistics \& Actuarial science \\
University of Kent\\
Canterbury, Kent CT2 7NF, United Kingdom\\
Email: {\tt S.Launois@kent.ac.uk}

\end{document}